\documentclass[english,times,12pt]{elsarticle}
\usepackage{color}
\usepackage{amsthm}
\usepackage{amsmath}
\usepackage{amssymb}
\usepackage{graphicx}
\usepackage{dcolumn}

\makeatletter

\providecommand{\tabularnewline}{\\}

\theoremstyle{plain}
\newtheorem{thm}{\protect\theoremname}
\theoremstyle{plain}

\theoremstyle{plain}
\newtheorem{lem}[thm]{\protect\lemmaname}
\ifx\proof\undefined
\newenvironment{proof}[1][\protect\proofname]{\par
\normalfont\topsep6\p@\@plus6\p@\relax
\trivlist
\itemindent\parindent
\item[\hskip\labelsep
\scshape
#1]\ignorespaces
}{%
\endtrivlist\@endpefalse
}
\providecommand{\proofname}{Proof}
\fi
\theoremstyle{plain}

\usepackage{babel}

\providecommand{\lemmaname}{Lemma}
\providecommand{\propositionname}{Proposition}
\providecommand{\theoremname}{Theorem}

\newtheorem{proposition}{Proposition}[section]
\newtheorem{theorem}[proposition]{Theorem}
\newtheorem{corollary}[proposition]{Corollary}
\newtheorem{lemma}[proposition]{Lemma}
 
\newtheorem{remark}[proposition]{Remark}

\usepackage{a4wide}
\usepackage{amsmath}
\usepackage{amsfonts}
\usepackage{amssymb}
\usepackage{enumerate}

\usepackage{graphicx}        %
\newcommand{\R}{\mathbb{R}}
\newcommand{\N}{\mathbb{N}}

\newcommand{\im}{{\rm im}\,}
\renewcommand{\ker}{{\rm ker}\,}

\newcommand{\rank}{{\rm rank}\,}

\newcommand{\dom}{{\rm dom}\,}

\newcommand{\argmin}{\operatorname{argmin}}

\begin{document}

\title{A least-squares collocation method for nonlinear higher-index differential-algebraic equations}

\author[KTH]{Michael Hanke}

\ead{hanke@nada.kth.se}

\address[KTH]{KTH Royal Institute of Technology, School of Engineering Sciences, 10044 Stockholm, Sweden}

\author[HUB]{Roswitha M{\"a}rz}

\ead{maerz@math.hu-berlin.de}

\address[HUB]{Humboldt University of Berlin, Institute of Mathematics, D-10099
Berlin, Germany}

\begin{abstract}
We introduce a direct numerical treatment of nonlinear higher-index differential-algebraic equations by means of overdetermined polynomial least-squares collocation. The procedure is not much more computationally expensive than standard collocation methods for regular  ordinary differential equations. The numerical experiments show impressive results. In contrast, the theoretical basic concept turns out to be considerably challenging.  So far, quite recently convergence proofs for linear problems have been published. In the present paper we come up to a first  convergence  result for nonlinear problems.
\end{abstract}

\maketitle

\section{Introduction}\label{s.Introduction}

For regular ordinary differential equations and index-1 differential-algebraic equations \textit{standard collocation methods} which rely on closed discretized systems\footnote{The number of unknowns equals the number of equations.} are known to work well. Moreover, Hessenberg form index-2 differential-algebraic equations can be treated successfully by so-called \textit{projected collocation methods} that complement standard collocation with an additional updating of the differential solution component by a projection step. This goes along with the well-posedness of the related initial and boundary value problems in natural settings; we refer to  \cite{LMW} for a  detailed survey. In contrast,  higher-index differential-algebraic equations lead to ill-posed\footnote{More precisely: Essentially ill-posed in Tichonov's sense, that is, the related operators feature nonclosed ranges.} initial and boundary value problems, and standard collocation methods  necessarily fail unless an elaborate index-reducing 
preprocessing is incorporated, which utilizes derivative array systems.

Recently (\cite{HMT,HMTWW}) first promising experiments concerning an least-squares overdetermined polynomial collocation directly applied to the DAE without any preprocessing have been reported. The theoretical justification appears to be quite challenging. So far, only sufficient convergence conditions are obtained for linear problems \cite{HMT,HM,HMTWW}. In the present paper we provide a first proof for nonlinear problems.
\medskip

The paper is organized as follows: In Section~\ref{s.Technicalities} we state the problem in detail. Then we provide a Hilbert space setting in Section \ref{s.Hilbert_setting}. This setting is more comfortable for the treatment of the given ill-posed problems. In Section \ref{s.single_part}, we introduce 
and investigate a kind of Newton-iteration related to a single 
partition, which uses bounded outer inverses as discussed in \cite{NaCh93} 
and which serves in the end as background for the Gauss-Newton iteration 
applied to an overdetermined collocation system. Then, we consider nested multiple partitions to ensure convergence of the iteration-projection method in Section \ref{s.Iteration}. The examples in Section \ref{s.E}  confirm the capability of the approach, but, having said that, they also indicate that our sufficient convergence conditions seem to be too unsubtle still. Finally, we provide some remarks and conclusions.

We use the symbol $\rVert\cdot\lVert$ for different function and operator norms. In general, in the given  context things will be unambiguous. Only on certain places, to prevent maybe imminent confusions we indicate the special norms by the corresponding subscripts, e.g., $\rVert\cdot\lVert_{L^{2}}$.

\section*{Some notations and abbreviations}

\begin{tabular}{ll}
$\R$& set of real numbers\\
$\mathcal L(\R^s,\R^n)$& space of linear operators from $\R^s$ to $\R^n$, also set of $n\times s $ - matrices with real entries\\ 
$\mathcal C([a,b], \R^{m})$& space of continuous functions mapping $[a,b]$ into $\R^{m}$\\
$\mathcal C^{s}([a,b], \R^{m})$& space of $s$-times continuously differentiable functions mapping $[a,b]$ into $\R^{m}$\\
$L^{2}:=L^{2}((a,b), \R^{m})$& Lebesque space of  functions mapping $(a,b)$ into $\R^{m}$\\
$H^{k}:=H^{k}((a,b), \R^{m})$& $:= W^{k,\,2}((a,b), \R^{m})$, Sobolev space of functions mapping $(a,b)$ into $\R^{m}$\\
$H^{1}_{D}:=H^{1}_{D}((a,b), \R^{m})$& $:=\{x\in L^{2} :\,Dx\in H^1$\} \\
$K^-$& generalized inverse of the operator $K$:\; $KK^-K=K$, $K^-KK^-=K^-$\\
$K^+$& Moore-Penrose inverse of $K$\\
$\ker K$& nullspace (kernel) of  $K$\\
$\im K$&  range (image) of $K$\\
$\langle \cdot,\cdot \rangle$& Euclidean inner product in $\R^m$\\
$( \cdot,\cdot )$& inner product in a function space\\
$|\cdot|$& Euclidean vector norm and spectral norm of a matrix\\
$\rVert\cdot\lVert$& norm of function space element and operator norm\\
$\oplus$& topological direct sum\\
$\mathfrak P_{N}$& set of all polynomials of degree less than or equal to $N$\\
DAE, DAO&differential-algebraic equation, differential-algebraic operator\\
ODE&ordinary differential equation\\
IVP, BVP&initial value problem, boundary value problem
\end{tabular}

\section{The issue and basic technicalities}\label{s.Technicalities}

We deal with IVPs and BVPs given in the form
\begin{align}
 f((Dx)'(t), x(t),t)&=0,\quad t\in [a, b],\label{DAE}\\
 g(x(a),x(b))&=0,\label{BC}
\end{align}
with $[a,b]$ being a compact interval,  $D=[I\; 0]\in \mathcal L(\R^{m},\R^{k})$, $\rank D=k$, and data $f(y,x,t)\in\R^{m}, y\in \R^{k}, x\in \mathcal D_{x}\subseteq \R^{m}, t\in \mathcal D_{t}\subset\R$, $ \mathcal D_{t}\supset [a,b]$, $g(u,v)\in \R^{l}, u,v\in\R^{m}$. The functions $f$ and $g$ are supposed to be at least continuous together with their partial derivatives $f_{y}, f_{x}, g_{u},g_{v}$.

We assume that the BVP (\ref{DAE}), (\ref{BC}) has the solution $x_{*}:[a,b]\rightarrow\R^{m}$ to be approximated. $x_{*}$ is supposed to be continuous  with continuously differentiable part $Dx_{*}$. Later on, among others for obtaining convergence orders, additional smoothness will be required.

Moreover, the DAE (\ref{DAE}) is supposed to be regular with (tractability) index $\mu\in\N$ and characteristics 
$0<r_{0}\leq\cdots\leq r_{\mu-1}<r_{m}=m$ around $x_{*}$, that means, the graph $\{(x_{*}(t),t):t\in [a,b]\}$ resides within a regularity region having these characteristics (e.g., \cite[Definition 3.28]{CRR}).
Note that then the derivative $(Dx)'$ is properly involved in the DAE (\ref{DAE}) so that $f_{y}(y,x,t)$ has full column-rank $k$.

Furthermore, in condition (\ref{BC}), we apply $l=m-\sum_{i=0}^{\mu-1}(m-r_{i})\geq 0$ which is the dynamical degree of freedom of the DAE. Recall that regular ODEs are indicated by $l=k=m$, regular index-1 DAEs by $l=k<m$, but higher-index DAEs by $l<k<m$. We are mainly interested in the last case. We further suppose the function $g$ to satisfy the relation
\begin{align}\label{g}
 g(u,v)=g(D^{+}Du,D^{+}Dv), \quad u,v\in R^{m},
\end{align}
so that the initial or boundary condition (\ref{BC}) actually applies to the differentiable component $Dx$ only. 

Together with the BVP (\ref{DAE}),(\ref{BC}) we consider the linear BVP,
\begin{align}
 A_{*}(t)(Dz)'(t)+B_{*}(t)z(t)&=q(t),\quad t\in [a,b],\label{linDAE}\\  
  G_{*\,a}z(a)+G_{*\,b}z(b)&=d,\label{linBC}
\end{align}
with
\begin{align*}
 &A_{*}(t):=f_{y}((Dx_{*})'(t),x_{*}(t),t),\quad B_{*}(t):=f_{x}((Dx_{*})'(t),x_{*}(t),t),\quad t\in [a,b],\\
 &G_{*\,a}:=g_{u}(x_{*}(a),x_{*}(b)),\quad G_{*\,b}:=g_{v}(x_{*}(a),x_{*}(b)).
\end{align*}
We assume the solution $x_{*}$ and possibly the data $f$ to be sufficiently smooth so that the linearized  DAE (\ref{linDAE}) is fine in the sense of \cite[Section 2.6]{CRR}. 
Since the solution $x_{*}$ resides in a regularity region of the DAE (\ref{DAE}), the linear DAE (\ref{linDAE}) inherits the characteristic values and the index $\mu$ of the nonlinear DAE, see \cite[Page 279]{Mae2014}. Furthermore, owing to condition (\ref{g}) it holds that
\begin{align}\label{G}
 \ker D\subseteq \ker G_{*\,a},\quad \ker D\subseteq \ker G_{*\,b}.
\end{align}
Condition (\ref{BC}) is supposed to be stated in such a way that the linear BVP (\ref{linDAE}),(\ref{linBC}) features accurately stated boundary condition in the sense of \cite[Definition 2.3]{LMW}), meaning that the problems 
\begin{align}\label{homlinDAE}
 A_{*}(t)(Dz)'(t)+B_{*}(t)z(t)=0,\; t\in [a,b],\quad
  G_{*\,a}z(a)+G_{*\,b}z(b)=d,
\end{align}
are uniquely solvable for each $d\in \R^{l}$, and the solutions satisfy the inequality 
\[
 \max_{t\in [a,b]}\,\rvert z(t)\lvert \leq \kappa_{BC}\, \rvert d\lvert,
\]
with a constant $\kappa_{BC}$. In particular, the homogeneous linear BVP, that is, the so-called variational problem, has then the trivial solution only.
\bigskip

Given the partition
\begin{align}\label{partition}
 \pi: a=t_{0}<t_{1}<\cdots<t_{n}=b,
\end{align}
with stepsizes $h_{j}=t_{j}-t_{j-1}$, maximal stepsize $h_{\pi}$, and minimal stepsize $h_{\pi,\min}$.
Denote by $\mathcal M_{[r]}$ the set of all partitions $\pi$ the ratio  of the maximal stepsize by the minimal stepsize of which is uniformly bounded by the constant $r<\infty$. 

Let $\mathcal C_{\pi}([a,b],\R^{m})$ denote the space of piecewise continuous functions having breakpoints merely at the mesh points.

Next we fix a number $N\geq 1$ and introduce the space $X_\pi$ of ansatz functions to approximate the solution $x_{*}$ by piecewise polynomial functions,
\begin{align}\label{ansatz}
 X_{\pi}=\{&x\in \mathcal C_{\pi}([a,b],\R^{m}):Dx\in \mathcal C([a,b],\R^{m}),\nonumber\\
 &x_{\kappa}\lvert_{[t_{j-1},t_{j})}\in \mathfrak P_{N},\, \kappa=1,\ldots,k,\quad
 x_{\kappa}\lvert_{[t_{j-1},t_{j})}\in \mathfrak P_{N-1},\, \kappa=k+1,\ldots,m,\; j=1,\ldots,n\}.
\end{align}
This ansatz space has dimension $nNm+k$. 
Choosing values 
\begin{align*}
 0<\tau_{1}<\cdots<\tau_{M}<1
\end{align*}
we specify $M$ collocation points per subinterval, i.e.,
\begin{align*}
 t_{ji}=t_{j-1}+\tau_{i}h_{j},\quad i=1,\ldots,M,\;j=1,\ldots,n,
\end{align*}
and are then confronted with the collocation system of $nMm+l$ equations for providing an approximation $x\in X_{\pi}$, namely,
\begin{align}
 f((Dx)'(t_{ji}), x(t_{ji}),t_{ji})&=0,\quad i=1,\ldots,M,\;j=1,\ldots,n,\label{collDAE}\\
 g(x(t_{0}),x(t_{n}))&=0,\label{collBC}.
\end{align}

The choice $M=N$ corresponds to the standard polynomial collocation yielding $nNm+l$ equations, which works well for regular ODEs and index-1 DAEs, with dynamical degree $l=k=m$ and $l=k<m$, respectively (cf.~\cite{LMW}).
In contrast, higher-index DAEs feature always a dynamical degree $0\leq l<k<m$. As it is well-known, completing the collocation system by additional $k-l$ consistent boundary conditions does not result in a suitable method owing to the ill-posedness of the higher-index problem, e.g., \cite[Example 1.1]{HMT}. 
As a matter of course, the choice $M>N$ goes along with an overdetermined system (\ref{collDAE}),(\ref{collBC}) comprising more equations than unknowns. 

Here we always set $M>N$ and  treat the overdetermined collocation system in a least-squares sense. More precisely, 
let $R_{\pi,M}:\mathcal C_{\pi}([a,b],\R^{m})\rightarrow \mathcal C_{\pi}([a,b],\R^{m})$ denote the restriction operator which assigns to $w\in \mathcal C_{\pi}([a,b],\R^{m})$  the piecewise polynomial $R_{\pi,M}w\in \mathcal C_{\pi}([a,b],\R^{m})$ of degree less than $M$ such that the interpolation conditions,
\begin{align*}
 (R_{\pi,M}w)(t_{j i})=w(t_{j i}),\quad i=1,\cdots,M,\; j=1,\cdots, n,
\end{align*}
are satisfied.
We also assign to $w\in \mathcal C_{\pi}([a,b],\R^{m})$ the vector $W\in \R^{mMn}$,
\begin{align*}
 W=\begin{bmatrix}
    W_{1}\\\vdots\\W_{n}
   \end{bmatrix}\in \R^{mMn},\quad W_{j}=\left(\frac{h_{j}}{M}\right)^{1/2}
   \begin{bmatrix}
    w(t_{j 1})\\\vdots\\w(t_{j M})
   \end{bmatrix}\in \R^{mM},
\end{align*}
which yields (cf.~\cite[Subsection 3]{HMTWW})
\begin{align*}
 \rVert R_{\pi,M}w\lVert_{L^{2}}^{2}= W^{T}\mathcal L W, \quad w\in \mathcal C_{\pi}([a,b],\R^{m}),
\end{align*}
with a positive definite, symmetric matrix $\mathcal L$. The entries of $\mathcal L$ do not at all depend on the partition $\pi$. They are fully determined by the corresponding $M$ Lagrange basis polynomials.
\medskip

Letting $w_{f}(t)=f((Dx)'(t),x(t),t), \; t\in [a,b]$, we introduce the functional 
\begin{align}
\psi_{\pi,M}(x)&=\;W^{T}_{f}\mathcal L W_{f}+\rvert g(x(a),x(b))\lvert^{2} \label{lsc}\\
&=\; \rVert R_{\pi,M}w_{f}\lVert^{2}+\rvert g(x(a),x(b))\lvert^{2},\quad x\in X_{\pi}\label{lscRF}.
\end{align}
The overdetermined least-squares collocation means now that we seek an element $\tilde x_{\pi}$ making the value $\psi_{\pi,M}(\tilde x_{\pi})$ as small as possible.
Note that there are positive constants $c_{\mathcal L}$, $C_{\mathcal L}$ such that
\begin{align*}
 c_{\mathcal L} \rvert W\lvert^{2}=c_{\mathcal L} W^{T}W\leq W^{T}\mathcal L W\leq C_{\mathcal L}W^{T}W=C_{\mathcal L} \rvert W\lvert^{2},\quad W\in \R^{mMn},
\end{align*}
which justifies the labeling \textit{least squares collocation}.
We refer to \cite{HMTWW,HMT} for a number of promising numerical experiments, see also Section \ref{s.E}.
Expression (\ref{lsc}) serves to indicate the basic numerical procedure, whereas formula (\ref{lscRF}) suggests that the mathematics behind is closely related to special properties of the restriction operator $R_{\pi,M}$ on the one hand, but on the other hand, to the  problem to minimize the functional
\begin{align}
 \psi(x)=\;\rVert w_{f}\lVert^{2}+\rvert g(x(a),x(b))\lvert^{2}\quad \text{subject to } x\in X_{\pi},\label{lscF}
\end{align}
for which (\ref{lscRF}) serves as approximation.
We refer to \cite{HM} for properties of the restriction operator in this context. The objective of the present paper is to contribute to the background problem (\ref{lscF}).

\section{Hilbert space setting}\label{s.Hilbert_setting}
Following the ideas of \cite{HMTWW,HMT} concerning linear problems, we investigate also the nonlinear problem (\ref{DAE}),(\ref{BC}) described in Section \ref{s.Introduction} as operator equation $\mathcal Fx=0$ in a Hilbert space setting, which is most comfortable for treating ill-posed problems. Besides standard function spaces such as $L^{2}$, $H^{1}$, $\mathcal C$, etc., equipped with usual inner products and norms, we use the space
\begin{align*}
 H^{1}_{D}=H^{1}_{D}((a,b),\R^{m})=\{x\in L^{2}((a,b),\R^{m}):Dx\in H^{1}((a,b),\R^{k})\,\},
\end{align*}
equipped with the inner product
\begin{align*}
 (x,\bar x)_{H^{1}_{D}}=(x,\bar x)_{L^{2}}+((Dx)',(D\bar x)')_{L^{2}},\quad x,\bar x\in H^{1}_{D}.
\end{align*}
$H^{1}_{D}$ is a Hilbert space, \cite[Lemma 6.9]{Mae2014}. Owing to the continuous embedding $H^{1}((a,b),\R^{k})\hookrightarrow \mathcal C([a,b],\R^{k})$, e.g., \cite[Theorem 0.4]{AP1993}, $x\in H^{1}_{D}$ implies $Dx\in C([a,b],\R^{k})$, and it holds
\begin{align}\label{embedding}
 \rVert Dx\lVert_{\infty}\leq \kappa \rVert Dx\lVert_{H^{1}}\leq \kappa \rVert x\lVert_{H^{1}_{D}},\quad x\in H^{1}_{D}.
\end{align}

We introduce the nonlinear operators $F,F_{BC},$ and $\mathcal F$, 
\begin{align*}
 F&:\dom F\subseteq H^{1}_{D}\rightarrow L^{2},\quad
 F_{BC}:\dom F\subseteq H^{1}_{D}\rightarrow \R^{l},\quad
 \mathcal F:=(F,F_{BC}):\dom F\subseteq H^{1}_{D}\rightarrow L^{2}\times\R^{l},
\end{align*}
\begin{align}
(Fx)(t)&:=f((Dx)'(t),x(t),t),\;t\in(a,b),\quad x\in \dom F,\label{nDAO}\\
F_{BC}\,x&:=g(x(a),x(b)), \quad x\in \dom F,\label{nBC}\\
\mathcal Fx&:=(Fx,F_{BC}x). \quad x\in \dom F,\label{nComp}
\end{align}
as well as the linear operators 
$T,T_{BC},$ and $\mathcal T$, 
\begin{align*}
 T&: H^{1}_{D}\rightarrow L_{2},\quad
 T_{BC}: H^{1}_{D}\rightarrow \R^{l},\quad
 \mathcal T:=(T,T_{BC}): H^{1}_{D}\rightarrow L^{2}\times\R^{l},
\end{align*}
\begin{align}
(Tx)(t)&:=A_{*}(Dx)'+B_{*}x,\;t\in(a,b),\quad x\in H^{1}_{D},\label{linearizedDAO}\\
T_{BC}\,x&:=G_{*\,a}x(a)+G_{*\,b}x(b), \quad x\in H^{1}_{D},\label{linearizedBC}\\
\mathcal Tx&:=(Tx,T_{BC}x). \quad x\in H^{1}_{D}.\label{linearizedComp}
\end{align}
We are merely interested in the local behavior of $F$ and $\mathcal F$ and suppose 
\[\dom \mathcal F=\dom F= \mathfrak B(x_{*}, \rho)\subset H^{1}_{D}.
\]

Regarding condition ($\ref{g}$) as well as (\ref{embedding}), we find the operators $F_{BC}$ and $T_{BC}$  well defined. $F_{BC}$ is Fr{\'e}chet-differentiable, which can be checked by straightforward computation. In particular, $F_{BC}(x_{*})=T_{BC}$. Moreover, supposing the partial derivatives $g_{u}, g_{v}$ to be Lipschitz
continuous, there is a constant $L_{BC}$ such that
\begin{align*}
 \rVert F'_{BC}(x)-F'_{BC}(\bar x)\lVert_{H^{1}_{D}\rightarrow \R^{l}}&\leq L_{BC}\rVert x-\bar x\lVert_{H^{1}_{D}}, \quad x,\bar x\in \dom F,\\
 \rVert F'_{BC}(x)\lVert_{H^{1}_{D}\rightarrow \R^{l}}&\leq L_{BC}\;\rho+ \rVert T_{BC}\lVert_{H^{1}_{D}\rightarrow \R^{l}}, \quad x\in \dom F.
\end{align*}
The linear operators $T$ and $\mathcal T$ are obviously bounded.
The operator $F$ is closely related to a certain Nemyckij operator as Proposition \ref{p.Nemyckij} below  indicates.
In the convergence proofs we will need that $F$ and thus $\mathcal F$ are G\^{a}teaux-differentiable on their domain with uniformly bounded G\^{a}teaux-derivatives,
\begin{align}\label{Mderivative}
\rVert  F'(x)\lVert_{H^{1}_{D}\rightarrow L^{2}}\; &\leq C_{F},\quad x\in \dom F,\nonumber\\
 \rVert \mathcal F'(x)\lVert_{H^{1}_{D}\rightarrow L^{2}\times\R^{l}}\;\leq \rVert  F'(x)\lVert_{H^{1}_{D}\rightarrow L^{2}}\; + \rVert F'_{BC}(x)\lVert_{H^{1}_{D}\rightarrow \R^{l}}&\leq C_{\mathcal F},\quad x\in \dom F,\\
 C_{\mathcal F}&:=C_{F}+L_{BC}\,\rho+\rVert T_{BC}\lVert_{H^{1}_{D}\rightarrow \R^{l}}.\nonumber 
\end{align}
Proposition \ref{p.Nemyckij} provides sufficient conditions to justify these assumptions.
\medskip

Moreover, we will need the inequality
\begin{align}\label{X-contF}
 \rVert F'(x)-F'(\bar x)\lVert_{H^{1}_{D}\rightarrow L^{2}}\;\leq L_{F} h_{\pi}^{-1/2} \,\rVert x-\bar x\lVert_{H^{1}_{\pi}}, \quad x,\bar x\in \dom F\cap X_{\pi},\; \pi\in .\mathcal M_{[r]},
\end{align}
to be valid with a  constant $L_{F}$ for the G\^ateaux-derivative $F'$ where $X_\pi$ is given by \eqref{ansatz}. 
Proposition \ref{p.Nemyckij} provides conditions also for this property to hold. Having (\ref{X-contF}), we are provided with a constant $L$  such that
\begin{align}\label{X-cont}
 \rVert \mathcal F'(x)-\mathcal F'(\bar x)\lVert_{H^{1}_{D}\rightarrow L^{2}\times\R^{l}}\;&\leq (L_{F}h_{\pi}^{-1/2}+L_{BC}) \,\rVert x-\bar x\lVert_{H^{1}_{D}}\nonumber\\
 &\leq L h_{\pi}^{-1/2}\rVert x-\bar x\lVert_{H^{1}_{D}}, \quad x,\bar x\in \dom F\cap X_{\pi},\;\pi\in\mathcal M_{[r]}.
\end{align}
Note that $L_{F}$ and $L$ depend on the stepsize ratio $r$.
\bigskip

Now the BVP (\ref{DAE}),(\ref{BC}) is represented by the operator equation $\mathcal Fx=0$ and the least-squares functional (\ref{lscF}) we are mainly interested in reads now 
\begin{align}\label{lsfunctional}
 \psi(x)=\;\rVert\mathcal Fx\lVert^{2},\quad x\in \dom F.
\end{align}
By construction, one has $T=F'(x_{*})$ and $\mathcal T=\mathcal F'(x_{*})$. The equation $\mathcal F'(x_{*})z=0$ represents the homogeneous variational BVP (\ref{homlinDAE}), with $d=0$, which has the trivial solution only. Therefore, the operator $\mathcal F'(x_{*})$ is injective. At this place we emphasize again, that higher-index DAEs lead to ill-posed problems. In the context here this means that $\im F'(x_{*})$ and $\im \mathcal F'(x_{*})$ are nonclosed subsets in $L^{2}$ and $L^{2}\times\R^{l}$, respectively, see \cite[Theorem 2.4]{HMT}, and the inverse $\mathcal F'(x_{*})^{-1}$ is unbounded.
\begin{proposition}\label{p.Nemyckij}
 Let $f$ and $D$ be as described in Section \ref{s.Technicalities}, with $\mathcal D_{x}=\R^{m}$, $\mathcal D_{y}=\R^{k}$ and bounded partial derivatives $f_{y}$ and $f_{x}$.
 \begin{enumerate}[(i)]
  \item Then,  $x\in H^{1}_{D}$ implies $Fx\in L^{2}$, and $F$ is G\^ateaux-differentiable, with the G\^ateaux-derivative $F'(x)$,
\begin{align}
 & F'(x)z=A_{(x)}(Dz)'+B_{(x)}z, \quad z\in H^{1}_{D},\label{GatF}\\
  &A_{(x)}(t):=f_{y}((Dx)'(t),x(t),t),\quad B_{(x)}(t):=f_{x}((Dx)'(t),x(t),t),\quad \text{\rm a.e.}\; t\in (a,b).
 \end{align}
Moreover, $F'(x)$ is uniformly bounded.
  \item If, additionally, the partial derivatives $f_x$ and $f_y$ satisfy the inequalities
\begin{align*}
\lvert f_{y}(y_{1},x_{1},t)-f_{y}(y_{2},x_{2},)\rvert^{2} & \leq \tilde L^2(|y_{1}-y_{2}\rvert^{2}+\lvert x_{1}-x_{2}\rvert^{2}),\\
\lvert f_{x}(y_{1},x_{1},t)-f_{x}(y_{2},x_{2},)\rvert^{2} & \leq \tilde L^2(|y_{1}-y_{2}\rvert^{2}+\lvert x_{1}-x_{2}\rvert^{2}),
\end{align*}
for all $x_1,x_2\in\R^m$ and $y_1,y_2\in\R^k$,
then there is a constant $L_F=L_F(r)$ such that
 \begin{align*}
 \rVert F'(x)-F'(\bar x)\lVert_{H^{1}_{D}\rightarrow L^{2}}\;\leq L_{F}h_\pi^{-1/2} \,\rVert x-\bar x\lVert_{H^{1}_{D}}, \quad x,\bar x\in  \dom F \cap X_{\pi},\; \pi\in \mathcal M_{[r]},
\end{align*} 
that is \eqref{X-contF}.
 \end{enumerate}
\end{proposition}
\begin{proof}
(i) Consider the operators $J:H^1_D\rightarrow L^2((a,b),\R^k)\times L^2$ given by $Jx=((Dx)',x)$ and 
$\tilde F:L^2((a,b),\R^k)\times L^2\rightarrow L^2$ defined as the Nemytskij operator
\[
 F(y,x)(t)=f(y(t),x(t),t).
\]
Under the stated conditions on $f$, $\tilde F$ is well-defined \cite[Theorem 1.2.2]{AP1993}. Moreover, it is
G\^ateaux-differentiable and its G\^ateaux-differential is given by \cite[Theorem 1.2.7]{AP1993}
\[
 \tilde F'(y,x)(u,v) = f_y(y,x,\cdot)u+f_x(y,x,\cdot)v.
\]
Now, $F=\tilde F\circ J$. Hence,
\begin{align*}
 \lim_{h\rightarrow 0} \frac{1}{h}\bigl(F(x+tz)-F(x)\bigr)
   &= \lim_{h\rightarrow 0} \frac{1}{h}\bigl(\tilde F(J(x+tz))-\tilde F(J(x))\bigr) \\
   &= \lim_{h\rightarrow 0} \frac{1}{h}\bigl(\tilde F(J(x)+tJ(z))-\tilde F(J(x))\bigr) \\
   &= \lim_{h\rightarrow 0} \frac{1}{h}\bigl(\tilde F(u+tv)-\tilde F(u)\bigr) \\
   &= \tilde F'(u)v \\
   &= f_y((Dx)',x,\cdot)(Dz)'+f_x((Dx)',x,\cdot)z
\end{align*}
where we used $u=J(x)$ and $v=J(z)$.

The norm of the derivative can be estimated by
\begin{align*}
 \lVert A_{(x)}(Dz)'+B_{(x)}z \rVert^2 
& \leq \lVert A_{(x)} \rVert_{L^\infty((a,b),\R^{m\times k})}^2 \lVert (Dz)' \rVert^2 +
       \lVert B_{(x)} \rVert_{L^\infty((a,b),\R^{m\times m})}^2 \lVert z \rVert^2 \\
& \leq C^2 \lVert z \rVert_{H^1_D}^2,
\end{align*}
where $C$ denotes a bound on the partial derivatives $f_y$ and $f_x$. Hence, the G\^ateaux-derivative
is uniformly bounded.

(ii) We will need an inverse inequality for functions from $X_\pi$. A consequence of
\cite[Theorem 3.2.6]{Ciarlet02} is the estimate
\begin{equation}\label{invineq}
 \lVert x \rVert_{L^\infty((a,b),\R^m)} \leq c h_\pi^{-1/2} \lVert x \rVert,\quad x\in X_\pi
\end{equation}
for a constant $c$ independent of $\pi\in\mathcal M_{[r]}$.

Let $\pi\in\mathcal M_{[r]}$ and $x,\bar x\in X_\pi$. Then ist holds
\begin{align*}
\lVert(F'(x)-F'(\bar{x}))z\rVert^{2} & =\int_{a}^{b}\lvert f_{y}((Dx)'(t),x(t),t)-f_{y}((D\bar{x})'(t),\bar{x}(t),t)\rvert{}^{2}\lvert(Dz)'(t)\rvert^{2}dt\\
 &\quad +\int_{a}^{b}\lvert f_{x}((Dx)'(t),x(t),t)-f_{x}((D\bar{x})'(t),\bar{x}(t),t)\rvert^{2}\lvert z(t)\rvert^{2}dt\\
 & \leq \tilde L^2\int_{a}^{b}\left(\lvert(Dx)'(t)-(D\bar{x})'(t)\rvert^{2}+\lvert x(t)-\bar{x}(t)\rvert^{2}\right)\lvert(Dz)'(t)\rvert^{2}dt\\
 &\quad +\tilde L^2\int_{a}^{b}\left(\lvert(Dx)'(t)-(D\bar{x})'(t)\rvert^{2}+\lvert x(t)-\bar{x}(t)\rvert^{2}\right)\lvert z(t)\rvert^{2}dt\\
 & \leq \tilde L^2\max_{a\leq t\leq b}\left(\lvert(Dx)'(t)-(D\bar{x})'(t)\rvert^{2}+\lvert x(t)-\bar{x}(t)\rvert^{2}\right)\int_{a}^{b}\left(\lvert(Dz)'(t)\rvert^{2}+\lvert z(t)\rvert^{2}\right)dt\\
 & \leq \tilde L^2\left(\lVert(Dx)'-(D\bar{x})'\rVert_{\infty}^{2}+\lVert x-\bar{x}\rVert_{\infty}^{2}\right)\lVert z\rVert_{H_{D}^{1}}^2.
\end{align*}
Applying \eqref{invineq}, we arrive at
\[
\lVert(F'(x)-F'(\bar{x}))z\rVert^{2}\leq \tilde L^2 c^{2}h^{-1}_\pi\lVert x-\bar{x}\rVert_{H_{D}^{1}}^{2}\lVert z\rVert_{H_{D}^{1}}^{2}
\]
which proves the assertion.
\end{proof}

\begin{remark}
According to Propsition~\ref{p.Nemyckij}, the G\^ateaux-derivative $F'$ is contiuous on each $X_\pi$. Hence, it is Fr\'echet-differentiable there. However, $F$ is in general not Fr\'echet-differentiable on $H^1_D$ unless it has a very special structure. A discussion of related question can be found in \cite[Section 1.2]{AP1993}.
\end{remark}

\begin{corollary}\label{xstar}
 Let the partial derivatives $f_{y}$ and $f_{x}$ satisfy the Lipschitz condition in Proposition \ref{p.Nemyckij}(ii) locally. Let $\tilde x\in \dom F$ be a sufficiently smooth function, possibly not belonging to $X_{\pi}$, $\lVert \tilde x-x_{*}\rVert\leq \rho/2$. Then it holds, for all $\tau\in[0,1]$,
\begin{equation}\label{X-gen}
 \lVert F'(x)-F'(x+(1-\tau)(\tilde x-x))\rVert \leq L_F h_\pi^{-1/2}(1-\tau)\lVert\tilde x-x\rVert_{H_D^1}+
\hat Lh_\pi^{N-1/2},
\; x\in X_\pi\cap\dom F,\;\pi\in \mathcal M_{[r]},
\end{equation}
with a constants $\hat L$. In particular, for $\tau=0$, we obtain
\[
 \lVert F'(x)-F'(\tilde x)\rVert \leq L_F h_\pi^{-1/2}\lVert\tilde x-x\rVert_{H_D^1}+
\hat Lh_\pi^{N-1/2},
\quad x\in X_\pi\cap\dom F, \;\pi\in \mathcal M_{[r]}
\]
and  
\begin{align*}
  \lVert \mathcal F'(x)-\mathcal F'(\tilde x)\rVert \leq L h_\pi^{-1/2}\lVert\tilde x-x\rVert_{H_D^1}+
\hat Lh_\pi^{N-1/2},
\quad x\in X_\pi\cap\dom F, \;\pi\in \mathcal M_{[r]}.
\end{align*}  
\end{corollary}
\begin{proof}
Let $I_\pi:H_D^1\cap C_{\pi}([a,b],\R^m)\rightarrow X_\pi$ be a piecewise polynomial interpolation operator. In order to be specific,
consider node sequences
\begin{gather*}
0=\sigma_0^d < \sigma_1^d < \cdots < \sigma_N^d = 1, \\
0<\sigma_1^a < \sigma_2^a < \cdots < \sigma_N^a < 1,
\end{gather*}
and define $I_{\pi}$ componentwise.
For a component $x_\kappa\in C[a,b]$, $1\leq \kappa\leq k$, $I_{\pi,\kappa}x_\kappa$ is the piecewise 
polynomial interpolation using the nodes $\bar t_{ji} = t_{j-1}+\sigma_i^dh_j$, $i=0,\ldots,N$, $j=1,\ldots,n$.
Analogously, for $x_\kappa\in C_\pi[a,b]$, $I_{\pi,\kappa}$, $k<\kappa\leq m$ is the piecewise polynomial iterpolation using the nodes $\bar t_{ji} = t_{j-1}+\sigma_i^ah_j$, $i=1,\ldots,N$, $j=1,\ldots,n$. Then we set
$I_\pi = [I_{\pi,1},\ldots,I_{\pi,m}]^T$.

Let $R_\pi=\tilde x-I_\pi\tilde x$ be the remainder. Standard interpolation results provide the estimate
\[
 \lVert (DR_\pi)' \rVert_\infty\leq C h_\pi^N,\quad \lVert R_\pi \rVert_\infty\leq C h_\pi^N, \quad \lVert R_{\pi}\rVert_{H_{D}^{1}}\leq (2(b-a))^{1/2} C\; h_{\pi}^{N}.
\]
For all  sufficiently fine partitions $\pi\in \mathcal M_{[r]}$, $I_{\pi}\tilde x$ belongs also to $\dom F$.

Since $I_\pi$ is the identity on $X_\pi$, we have, for each $x\in X_{\pi}\cap \dom F$,
\begin{align*}
x+(1-\tau)(\tilde x-x)-I_\pi(x+(1-\tau)(\tilde x-x))
 &= x+(1-\tau)(\tilde x-x)-(x+(1-\tau)(I_\pi\tilde x-x)) \\
 &= (1-\tau)R_\pi.
\end{align*}
Following the lines of the proof of Proposition~\ref{p.Nemyckij}(ii) we arrive at the estimate
\begin{align*}
 \lVert(F'(I_\pi(x+(1-\tau)(\tilde x-x)))-F'(x+(1-\tau)(\tilde x-x)))z\rVert^{2} &\leq \tilde L^2 \bigl(\lVert (DR_\pi)' \rVert_\infty^2+
  \lVert R_\pi \rVert_\infty^2\bigr)\lVert z \rVert_{H_{D}^{1}}^2 \\
&\leq 2\tilde L^2 C^2 h_\pi^{2N}\lVert z \rVert_{H_{D}^{1}}^2,
\end{align*}
hence
\[
 \lVert F'(I_\pi(x+(1-\tau)(\tilde x-x)))-F'(x+(1-\tau)(\tilde x-x))\rVert \leq \tilde{\tilde L} h_\pi^N.
\]
Then we obtain
\begin{align*}
\lVert F'(x)-F'(x+(1-\tau)(\tilde x-x))\rVert
  &\leq   \lVert F'(x)-F'(I_\pi(x+(1-\tau)(\tilde x-x)))\rVert \\
  &\quad  + \lVert F'(I_\pi(x+(1-\tau)(\tilde x-x)))-F'(x+(1-\tau)(\tilde x-x))\rVert \\
  &\leq L_F h_\pi^{-1/2}(1-\tau)\lVert I_\pi\tilde x-x\rVert_{H_D^1}+\tilde{\tilde L} h_\pi^N \\
  &\leq L_F h_\pi^{-1/2}(1-\tau)\bigl(\lVert\tilde x-x\rVert_{H_D^1}+Ch_\pi^N\bigr)+\tilde{\tilde L} h_\pi^N.
\end{align*}
This proofs the assertion.
\end{proof}

\section{Properties related to individual sufficiently fine partitions $\pi$}\label{s.single_part}

This section is to provide an approximation of the solution $x_{*}$ by means of an iteration residing in  $X_{\pi}$ for an arbitrary sufficiently fine individual partition $\pi\in \mathcal M_{[r]}$.

The space of ansatz functions $X_{\pi}$ is defined by (\ref{ansatz}) as before.
Below we frequently apply the topological decompositions
\begin{align}\label{decomp}
 H^{1}_{D}= X_{\pi}\oplus X_{\pi}^{\bot},\quad L^{2}= F'(x)X_{\pi}\oplus (F'(x)X_{\pi})^{\bot},\quad L^{2}\times \R^{l}= \mathcal F'(x)X_{\pi}\oplus (\mathcal F'(x)X_{\pi})^{\bot}.
\end{align}
and the associated orthoprojectors 
\begin{align}\label{proj}
 U_{\pi}:H^{1}_{D}\rightarrow H^{1}_{D}, \quad V_{\pi}(x):L^{2}\rightarrow L^{2},\quad \mathcal V_{\pi}(x):L^{2}\times \R^{l}\rightarrow L^{2}\times\R^{l},\\
 \im U_{\pi}= X_{\pi},\quad \im V_{\pi}(x)= F'(x)X_{\pi},\quad \im \mathcal V_{\pi}(x)= \mathcal F'(x)X_{\pi},
\end{align}
in which $x\in \dom F$. $F'(x_{*})$ is a fine DAO with index $\mu$ and  $\mathcal F'(x_{*})$ is injective, but its inverse is unbounded if $\mu>1$.
\begin{lemma}\label{l.basic}
Let $x_{*}$ be sufficiently smooth. Let $N>\mu-1$.
Choose $s\in\R$ with $s>\mu-1/2>0$ and $\rho_{\pi}:=c_{\rho}h^{s}_{\pi}$, with a constant $c_{\rho}>0$.
Then there is a constant $c_{\gamma} >0$, such that the following relations become valid: 
\begin{align}
 \rVert\; (\mathcal F'(x_{*})U_{\pi})^{+} \lVert \leq \Gamma_{\pi}:=\frac{1}{c_{\gamma}}h_{\pi}^{1-\mu},\label{MP1}\\
 \ker \mathcal F'(x)U_{\pi}=\ker U_{\pi}, \quad x\in \bar{\mathfrak B}(x_{*},\rho_{\pi})\cap X_{\pi},\label{MP2}\\
 (\mathcal F'(x)U_{\pi})^{+}=  (\mathcal V_{\pi}(x_{*})\mathcal F'(x)U_{\pi})^{+} \mathcal V_{\pi}(x_{*})\mathcal V_{\pi}(x), \quad x\in \bar{\mathfrak B}(x_{*},\rho_{\pi})\cap X_{\pi},\label{MP3}\\
  \rVert\; (\mathcal F'(x)U_{\pi})^{+} \lVert \leq \rVert\; (\mathcal V_{\pi}(x_{*})\mathcal F'(x)U_{\pi})^{+} \lVert\leq 2\Gamma_{\pi},\quad x\in \bar{\mathfrak B}(x_{*},\rho_{\pi})\cap X_{\pi},\label{MP4}
\end{align}
for each arbitrary mesh $\pi\in \mathcal M_{[r]}$ with sufficiently small $h_{\pi}$.
\end{lemma}
\begin{proof}
The existence of $c_{\gamma}>0$ as well as the inequality (\ref{MP1}) are ensured by \cite[Theorem 4.1]{HMT} concerning the instability threshold. $c_{\gamma}$ may depend on the ratio $r$. The injectivity of $\mathcal F'(x_{*})$ immediately implies $\ker \mathcal F'(x_{*})U_{\pi}=\ker U_{\pi}$.

For $x\in \bar{\mathfrak B}(x_{*},\rho_{\pi})\cap X_{\pi}$, $\rho_{\pi}<\rho$, we have
\begin{align}\label{MP5}
 \ker U_{\pi}\subseteq  \mathcal F'(x)U_{\pi}\subseteq  \mathcal V(x_{*})\mathcal F'(x)U_{\pi}
\end{align}
and 
\[
 \mathcal V(x_{*})\mathcal F'(x)U_{\pi}=\underbrace{\mathcal F'(x_{*})U_{\pi}}_{=:\mathfrak A}+\underbrace{\mathcal V(x_{*})(\mathcal F'(x)-\mathcal F'(x_{*}))U_{\pi}}_{=:\mathfrak E}=\mathfrak A+\mathfrak E..
\]
Making the stepsize $h_{\pi}$ small enough and regarding Corollary~\ref{xstar}, \eqref{Mderivative}, and (\ref{MP1}) yields
\begin{equation}
\rVert \mathfrak A^{+}\lVert \,\rVert \mathfrak E\lVert\,\leq \Gamma_{\pi}( L h_{\pi}^{-1/2}\rho_{\pi} + \hat L h_{\pi}^{N-1/2})\leq \frac{1}{c_{\gamma}}\left(c_{\rho} L h_{\pi}^{s-\mu+1/2}+\hat L h_{\pi}^{N-\mu+1/2}\right)
\leq \frac{1}{2}. \label{hsmall}
\end{equation}
Applying Lemma~A.\ref{ST} of the appendix it results that
\[
\dim \ker \mathcal V(x_{*})\mathcal F'(x)U_{\pi}= \dim \ker \mathcal F'(x_{*})U_{\pi},\quad \text{thus} \quad\ker \mathcal V(x_{*})\mathcal F'(x)U_{\pi}=\ker U_{\pi},
\]
and further 
\begin{align*}
 \rVert(\mathcal V(x_{*})\mathcal F'(x)U_{\pi})^{+} \lVert \leq \frac{\rVert \mathfrak A^{+}\lVert }{1-\rVert \mathfrak A^{+}\lVert \;\rVert \mathfrak E\lVert }\leq 2\Gamma_{\pi}.
\end{align*}
Taking into account (\ref{MP5}) we have
\begin{align*}
 \ker U_{\pi}=\ker \mathcal F'(x)U_{\pi}= \ker \mathcal V(x_{*})\mathcal F'(x)U_{\pi},
\end{align*}
and, in particular, (\ref{MP2}). It also follows that
\begin{align*}
 U_{\pi}=(\mathcal F(x)U_{\pi})^{+}\mathcal F(x)U_{\pi}, \quad U_{\pi}=(\mathcal V_{\pi}(x_{*})\mathcal F(x)U_{\pi})^{+}\,\mathcal V_{\pi}(x_{*})\mathcal F(x)U_{\pi}.
\end{align*}
Multiplying the last identity from the right by ($\mathcal F(x)U_{\pi})^{+}$ yields
\begin{align*}
 (\mathcal F(x)U_{\pi})^{+}=(\mathcal V_{\pi}(x_{*})\mathcal F(x)U_{\pi})^{+}\,\mathcal V_{\pi}(x_{*})\mathcal V_{\pi}(x),
\end{align*}
that means (\ref{MP3}), and (\ref{MP4}) follows immediately.
\end{proof}
It should be noted that $s$ in the previous lemma is not restricted to be an integer.

As previously agreed upon, there exists $x_{*}$ such that $\mathcal Fx_{*}=0$, thus $\psi(x_{*})=0$,  $F'(x_{*})$ is a fine DAO, the varionational problem $\mathcal F'(x_{*})z=0$ features accurately stated boundary condition, and the composed operator $\mathcal F'(x_{*})$ is injective.
Assuming the solution $x_{*}$ to be smooth enough we apply the estimates (cf.~\cite{HMT})
\begin{align}\label{alpha}
 \alpha_{\pi}:=\rVert U_{\pi}x_{*}-x_{*}\lVert \leq c_{\alpha} h_{\pi}^{N},
\end{align}
in which $N$ is again the polynomial degree used for the ansatz space $X_{\pi}$. 
\bigskip

Since the inverse $\mathcal F'(x_{*})^{-1}$ is unbounded, standard Newton-like iterations cannot be expected to work well here.
Instead we apply a kind of projected Newton iteration using the bounded Moore-Penrose inverse\footnote{Note that $(\mathcal F'(x)U_{\pi})^{+}$ is a bounded outer inverse of $\mathcal F'(x_{*})$.} $(\mathcal F'(x)U_{\pi})^{+}$ against the background of Lemma \ref{l.basic}. 

More precisely, supposing that $h_{\pi}$ is small enough, we take an initial guess $x_{0}\in \bar{\mathfrak B}(x_{*},\rho_{\pi})\cap X_{\pi}$  and provide the correction $z_{1}$ by means of the least-squares problem
\begin{align}\label{correction}
 z_{1}=\argmin \{\,\rVert \mathcal F'(x_{0})z+\mathcal Fx_{0}\lVert^{2}: z\in X_{\pi}\}
 =-(\mathcal F'(x_{0})U_{\pi})^{+}\,\mathcal Fx_{0},
\end{align}
and then put $x_{1}=x_{0}+z_{1}$, and so on. By construction, $z_{1}$ is well defined and belongs to $X_{\pi}$, and so does the new iteration $x_{1}$. Notice that $z_{1}=U_{\pi}z_{1}$ serves as descent direction of the functional $\psi$ at $x_{0}$,  as long as $\mathcal V_{\pi}(x_{0})\mathcal Fx_{0}\neq 0$, because of
\begin{align*}
 \psi'(x_{0})z_{1}&=2 (\mathcal F'(x_{0})z_{1}, \mathcal Fx_{0})=2 (\mathcal F'(x_{0})U_{\pi}z_{1}, \mathcal Fx_{0})=-2 (\mathcal F'(x_{0})U_{\pi}(\mathcal F'(x_{0})U_{\pi})^{+}\mathcal Fx_{0}, \mathcal Fx_{0})\\
 &=-2 (\mathcal V(x_{0})\mathcal Fx_{0}, \mathcal Fx_{0})=-2 \rVert\mathcal V(x_{0})\mathcal Fx_{0}\lVert^{2}.
\end{align*}
Next we ask if $x_{1}$ belongs to the ball $\bar{\mathfrak B}(x_{*},\rho_{\pi})$ .
For this aim we derive
\begin{align*}
 x_{1}-x_{*}&=x_{0}-x_{*}-(\mathcal F'(x_{0})U_{\pi})^{+}\,(\mathcal Fx_{0}-\mathcal Fx_{*})\\
 &=U_{\pi}(x_{0}-x_{*})-(I-U_{\pi})x_{*}-(\mathcal F'(x_{0})U_{\pi})^{+}\int_{0}^{1}\mathcal F'(\tau x_{0}+(1-\tau)x_{*}){\rm d\tau}\,(x_{0}-x_{*})\\
 &=\mathfrak B-\mathfrak D.
 \end{align*}
Then
 \begin{align*}
 \mathfrak B&=U_{\pi}(x_{0}-x_{*})-(\mathcal F'(x_{0})U_{\pi})^{+}\int_{0}^{1}\mathcal F'(\tau x_{0}+(1-\tau)x_{*}){\rm d\tau}\,U_{\pi}(x_{0}-x_{*})\\
 &=(\mathcal F'(x_{0})U_{\pi})^{+}\mathcal F'(x_{0})U_{\pi}(x_{0}-x_{*})-(\mathcal F'(x_{0})U_{\pi})^{+}\int_{0}^{1}\mathcal F'(\tau x_{0}+(1-\tau)x_{*}){\rm d\tau}\,U_{\pi}(x_{0}-x_{*})\\
 &=(\mathcal F'(x_{0})U_{\pi})^{+}\int_{0}^{1}(\mathcal F'(x_{0})-\mathcal F'(\tau x_{0}+(1-\tau)x_{*})\,){\rm d\tau}\,U_{\pi}(x_{0}-x_{*}),
\end{align*}
hence, applying Corollary \ref{xstar} for $\tilde x=\tau x_{0}+(1-\tau)x_{*}$, $\rVert \tilde x-x_{*}\lVert\leq \rho_{\pi}\leq \frac{1}{2}\rho$, and supposing $N>s$,
\begin{align*}
 \rVert\mathfrak B\lVert&\leq 2\Gamma_\pi\left( \frac{1}{2}L h_{\pi}^{-1/2}\rho_\pi+\hat L h^{N-1/2}\right) \rVert x_{0}-x_{*}\lVert \leq \frac{1}{c_{\gamma}}\left(L c_{\rho}h_{\pi}^{s-\mu+1/2}\rho_\pi+2\hat L h^{N-\mu+1/2}\right) \rVert x_{0}-x_{*}\lVert \\
  &\leq\frac{1}{2}\rVert x_{0}-x_{*}\lVert,
\end{align*}
for sufficiently small $h_\pi$, cf.~\eqref{hsmall}. Next, for 
\begin{align*}
 \mathfrak D&=(I-U_{\pi})x_{*}+(\mathcal F'(x_{0})U_{\pi})^{+}\int_{0}^{1}\mathcal F'(\tau x_{0}+(1-\tau)x_{*}){\rm d\tau}(I-U_{\pi})(-x_{*})\\
 &=\left\{I-(\mathcal F'(x_{0})U_{\pi})^{+}\int_{0}^{1}\mathcal F'(\tau x_{0}+(1-\tau)x_{*}){\rm d\tau}\right\}(I-U_{\pi})x_{*}
\end{align*}
we obtain a constant $c_{*}$ such that 
\begin{align*}
 \rVert\mathfrak D\lVert&\leq (1+\Gamma_\pi C_{\mathcal F}) \lVert x_{*}-U_\pi x_{*} \rVert\leq c_{\alpha}(\frac{2}{c_{\gamma}}C_{\mathcal F}+h_{\pi}^{\mu-1})h_{\pi}^{N-\mu+1}
\leq c_{*} h_{\pi}^{N-\mu+1}.
\end{align*}
Now, to ensure that  $x_{1}$ belongs to the ball $\bar{\mathfrak B}(x_{*},\rho_{\pi})$, we are confronted with the requirement
\begin{align*}
 \rVert\mathfrak D\lVert\leq c_{*}h_{\pi}^{N-\mu+1}\leq \frac{1}{2}  c_{\rho}h_{\pi}^{s},
\end{align*}
which becomes valid by choosing $N$ so that
\begin{align}\label{N}
 N-\mu+1>s,
\end{align}
for all sufficiently fine meshes $\pi\in \mathcal M_{[r]}$. Then we continue the iterations by providing 
\begin{align}
x_{k+1}&=x_{k}+z_{k+1},\label{xk}\\
 z_{k+1}&=\argmin \{\,\rVert \mathcal F'(x_{k})z+\mathcal Fx_{k}\lVert^{2}: z\in X_{\pi}\}
 =-(\mathcal F'(x_{k})U_{\pi})^{+}\,\mathcal Fx_{k},\quad k\geq 0.\label{zk}
\end{align}
The sequence $\{x_{k}\}$ remains in $\bar{\mathfrak B}(x_{*},\rho_{\pi})$. Furthermore we have
\begin{align*}
 \rVert x_{k+1}-x_{*}\lVert&\leq \frac{1}{2}\rVert x_{k}-x_{*}\lVert +c_{*}h_{\pi}^{N-\mu+1}\leq\cdots\leq \frac{1}{2^{k+1}}\rVert x_{0}-x_{*}\lVert +\sum_{i=0}^{k}\frac{1}{2^{i}} c_{*} h_{\pi}^{N-\mu+1}\\
 &\leq \left(\frac{1}{2}\right)^{k+1}\rVert x_{0}-x_{*}\lVert +2c_{*} h_{\pi}^{N-\mu+1}\leq \left(\frac{1}{2}\right)^{k+1} c_{\rho} h_{\pi}^{s} +2c_{*} h_{\pi}^{N-\mu+1},\quad k\geq 0.
\end{align*}
There is a number $k_{\pi}\in \N$ so that one has  $\left(\frac{1}{2}\right)^{k+1}\leq \frac{c_{*}}{c_{\rho}}h_{\pi}^{N-\mu+1-s}$ for all $k\geq k_{\pi}$, and hence 
\begin{align}\label{estimate}
 \rVert x_{k+1}-x_{*}\lVert\leq 3c_{*} h_{\pi}^{N-\mu+1},\quad k\geq k_{\pi}.
\end{align}
We summarize what we get:
\begin{theorem}\label{t.pi}
Let $\mathcal Fx=0$ denote the operator formulation from Section \ref{s.Hilbert_setting} associated with the BVP (\ref{DAE}),(\ref{BC}), $\mathcal Fx_{*}=0$, $\ker \mathcal F'(x_{*})=\{0\}$, and $x_{*}$ be sufficiently smooth for (\ref{alpha}) to hold.\\
Let the radius $\rho_{\pi}$ and the bound $\Gamma_{\pi}$ be as introduced in Lemma \ref{l.basic}, and
\begin{align}\label{N_and_s}
 N-\mu+1>s>\mu-1/2,
\end{align}
and the mesh $\pi\in\mathcal M_{[r]}$ be sufficiently fine. Then the iteration (\ref{xk}) starting from $x_{0}\in \bar{\mathfrak B}(x_{*},\rho_{\pi})\cap X_{\pi}$ remains therein and there is a number $k_{\pi}\in \N$ such that the estimate (\ref{estimate}) is valid and 
\begin{align}\label{estimate_psi}
 \psi(x_{k+1})\leq (3 c_{*} C_{\mathcal F})^{2} h_{\pi}^{2(N-\mu+1)},\quad k\geq k_{\pi}.
\end{align}
\end{theorem}
\begin{proof}
 It only remains to verify (\ref{estimate_psi}) which is a simple consequence of (\ref{Mderivative}) and (\ref{estimate}):
 \begin{align*}
 \psi(x_{k+1})=\rVert \mathcal F x_{k+1}\lVert^{2}=\rVert \mathcal F x_{k+1}-\mathcal F x_{*}\lVert^{2}\leq C_{\mathcal F}^{2}\;\rVert x_{k+1}-x_{*}\lVert^{2}\leq C_{\mathcal F}^{2}(3c_{*} h_{\pi}^{N-\mu+1})^{2},\quad k\geq k_{\pi}.
\end{align*}
\end{proof}
Let us emphasize that the constants $c_{\gamma},c_{\rho}, c_{\alpha}$, $c_{*}$, and $M_{*}$ are global bounds for all partitions $\pi\in\mathcal M_{[r]}$.
\section{Numerical experiments}\label{s.E}

In this section, we present the results of some experiments in order to illustrate the properties
of the proposed method.

The nonlinear least-squares method \eqref{lsfunctional} has been implemented in Matlab. Instead of \eqref{lsfunctional}, its approximation $\psi_{\pi,M}$ of \eqref{lscRF} has been used. The finite-dimensional problems have been solved using a Matlab implementation of a Gauss-Newton method following the lines
of \cite[Section 4.3]{De04}.
The iteration has been stopped if no further improvement in $\psi_{\pi,M}(x_k)$ could be observed.
For the purposes of investigating the convergence of the method, an interpolation of the exact solution has been used as an initial guess.

\subsection{The mathematical pendulum}

This problem has been used in many publications for demonstrating properties of algorithms for the solution
of differentail algebraic systems. We use the formulation
\begin{align*}
 x'' &= -x\lambda, \\
y'' &= -y\lambda-g, \\
0 &= x^2+y^2-L^2.
\end{align*}
The underlying interval is $(0,1)$. The parameters are chosen to be $g=16$, $L=\sqrt{8}$. We consider the initial values $y(0)=2$ and $y'(0)=0$. This problem has index 3. Therefore, the results of Theorem~\ref{t.pi} are
only valid if $N\geq 5$. For $N=5$, $s=\mu-1/4$ can be chosen. However, the expected orders are observed in all cases $N\geq2$. The case $N=1$ is rather surprising since we observed bounded solutions instead of diverging ones.

In Tables~\ref{T.pendel3e} and \ref{T.pendel3o} as well as Tables~\ref{T.pendel5e} and \ref{T.pendel5o} results for $N=3$ and $N=5$, respectively, are presented. In both cases, uniform grids and $M=N+1$ uniformly distributed collocation points per subinterval have been used.

\begin{table}
\caption{Errror in $L^{2}(0,1)$ for $N=3$ for the pendulum example. $n$ equidistant grid points and $M=N+1$
uniformly distributed collocation points have been used}\label{T.pendel3e}
\begin{center}
\begin{tabular}{rD{.}{.}{2.6}D{.}{.}{2.6}D{.}{.}{2.6}D{.}{.}{2.6}D{.}{.}{2.6}}\\ 
\hline
\multicolumn{1}{c}{$n$}  & \multicolumn{1}{c}{$x$} & \multicolumn{1}{c}{$x'$} & \multicolumn{1}{c}{$y$} & \multicolumn{1}{c}{$y'$} & \multicolumn{1}{c}{$\lambda$}\\ 
\hline
10 & 4.42e-02 & 1.17e-01 & 1.83e-02 & 1.01e-01 & 6.25e-01  \\ 
20 & 6.01e-03 & 1.76e-02 & 2.48e-03 & 1.98e-02 & 3.33e-01  \\ 
40 & 8.28e-04 & 3.07e-03 & 3.41e-04 & 4.47e-03 & 1.72e-01  \\ 
80 & 1.11e-04 & 6.26e-04 & 4.59e-05 & 1.07e-03 & 8.67e-02  \\ 
160 & 1.42e-05 & 1.44e-04 & 5.87e-06 & 2.64e-04 & 4.34e-02  \\ 
320 & 1.86e-06 & 3.50e-05 & 7.65e-07 & 6.58e-05 & 2.17e-02  \\ 
640 & 2.32e-07 & 8.68e-06 & 9.57e-08 & 1.64e-05 & 1.08e-02  \\ 
\hline
\end{tabular}
\end{center}
\end{table}

\begin{table}
\caption{Order estimate for $N=3$ for the pendulum example. $n$ equidistant grid points and $M=N+1$
uniformly
distributed collocation points have been used}\label{T.pendel3o}
\begin{center}
\begin{tabular}{rD{.}{.}{1.1}D{.}{.}{1.1}D{.}{.}{1.1}D{.}{.}{1.1}D{.}{.}{1.1}}\\ 
\hline
\multicolumn{1}{c}{$n$}  & \multicolumn{1}{c}{$x$} & \multicolumn{1}{c}{$x'$} & \multicolumn{1}{c}{$y$} & \multicolumn{1}{c}{$y'$} & \multicolumn{1}{c}{$\lambda$}\\ 
\hline
10 & 2.9 & 2.7 & 2.9 & 2.3 & 0.9  \\ 
20 & 2.9 & 2.5 & 2.9 & 2.1 & 1.0  \\ 
40 & 2.9 & 2.3 & 2.9 & 2.1 & 1.0  \\ 
80 & 3.0 & 2.1 & 3.0 & 2.0 & 1.0  \\ 
160 & 2.9 & 2.0 & 2.9 & 2.0 & 1.0  \\ 
320 & 3.0 & 2.0 & 3.0 & 2.0 & 1.0  \\ 
\hline
\end{tabular}
\end{center}
\end{table}

\begin{table}
\caption{Errror in $L^{2}(0,1)$ for $N=5$ for the pendulum example. $n$ equidistant grid points and $M=N+1$
uniformly distributed collocation points have been used}\label{T.pendel5e}
\begin{center}
\begin{tabular}{rD{.}{.}{2.6}D{.}{.}{2.6}D{.}{.}{2.6}D{.}{.}{2.6}D{.}{.}{2.6}}\\ 
\hline
\multicolumn{1}{c}{$n$}  & \multicolumn{1}{c}{$x$} & \multicolumn{1}{c}{$x'$} & \multicolumn{1}{c}{$y$} & \multicolumn{1}{c}{$y'$} & \multicolumn{1}{c}{$\lambda$}\\ 
\hline
10 & 4.13e-04 & 1.28e-03 & 1.75e-04 & 1.99e-03 & 3.61e-02  \\ 
20 & 4.59e-05 & 1.22e-04 & 1.88e-05 & 1.38e-04 & 4.90e-03  \\ 
40 & 1.45e-06 & 4.44e-06 & 5.94e-07 & 6.94e-06 & 6.18e-04  \\ 
80 & 3.43e-08 & 1.82e-07 & 1.41e-08 & 3.97e-07 & 7.74e-05  \\ 
160 & 1.02e-09 & 1.04e-08 & 4.17e-10 & 2.45e-08 & 9.68e-06  \\ 
320 & 5.57e-11 & 6.41e-10 & 2.28e-11 & 1.52e-09 & 1.21e-06  \\ 
\hline
\end{tabular}
\end{center}
\end{table}

\begin{table}
\caption{Order estimate for $N=5$ for the pendulum example. $n$ equidistant grid points and $M=N+1$
uniformly
distributed collocation points}\label{T.pendel5o}
\begin{center}
\begin{tabular}{rD{.}{.}{1.1}D{.}{.}{1.1}D{.}{.}{1.1}D{.}{.}{1.1}D{.}{.}{1.1}}\\ 
\hline
\multicolumn{1}{c}{$n$}  & \multicolumn{1}{c}{$x$} & \multicolumn{1}{c}{$x'$} & \multicolumn{1}{c}{$y$} & \multicolumn{1}{c}{$y'$} & \multicolumn{1}{c}{$\lambda$}\\ 
\hline
10 & 3.2 & 3.4 & 3.2 & 3.8 & 2.9  \\ 
20 & 5.0 & 4.8 & 5.0 & 4.3 & 3.0  \\ 
40 & 5.4 & 4.6 & 5.4 & 4.1 & 3.0  \\ 
80 & 5.1 & 4.1 & 5.1 & 4.0 & 3.0  \\ 
160 & 4.2 & 4.0 & 4.2 & 4.0 & 3.0  \\ 
\hline
\end{tabular}
\end{center}
\end{table}

\subsection{An example proposed by S.L.~Campbell and E.~Moore}

In \cite{CampbellMoore95}, the following system is used as an example:
\[
 \begin{aligned}
  x'_1-x_4&=0,\\
  x'_2-x_5&=0,\\
  x'_3-x_6&=0,\\
  x'_4-x_6\cos t+x_3\sin t+x_5-2x_1(1-r(x_1^2+x_2^2)^{-\frac{1}{2}})x_7&=0,\\
  x'_5-x_6\sin t-x_3\cos t-x_4-2x_2(1-r(x_1^2+x_2^2)^{-\frac{1}{2}})x_7&=0,\\
  x'_6+x_3-2x_3x_7&=0,\\
  x_1^2+x_2^2+x_3^2-2r(x_1^2+x_2^2)^{\frac{1}{2}}+r^2-\rho^2&=0.
 \end{aligned}\label{CaMoEx}
\]
The solution considered in the reference is
\begin{align*}
 x_{*1}&=(\rho\cos(2\pi-t)+r)\cos t=(\rho\cos t +r)\cos t,\\
 x_{*2}&=(\rho\cos(2\pi-t)+r)\sin t=(\rho\cos t +r)\sin t,\\
 x_{*3}&=\rho \sin(2\pi-t)=-\rho\sin t,
\end{align*}
yielding
\begin{align*}
 x_{*4}&=-(\rho\cos(2\pi-t)+r)\sin t + \rho \sin(2\pi-t)\cos t,\\
 x_{*5}&=(\rho\cos(2\pi-t)+r)\cos t + \rho \sin(2\pi-t)\sin t,\\
 x_{*6}&=-\rho \cos(2\pi-t),\\
 x_{*7}&=0.
\end{align*}
In \cite{CampbellMoore95}, the inequality $r>\rho$ is supposed and the numerical experiments are carried out for $\rho=5$ and $r=10$. We use the same parameters in the following experiment. Under these conditions, the
problem has index 3.

A thorough discussion as well as numerical experiments of the version linearized in the solution $x_{*}$
is given in \cite{HMT}. In order to stimulate discussions of the least-squares method for nonlinear
problems, also results for the original nonlinear version have been provided in this reference. We cite the
results in Tables~\ref{T.CaMoe} and \ref{T.CaMoo}. Theorem~\ref{t.pi} is only valid for $N\geq 5$ in this example and, thus, the corresponding order is strictly proven. However, the expected orders are observed in allowed cases $N\geq 2$. The case $N=1$ is rather surprising besause we observe bounded solutions even if we expecteddiverging ones.

\begin{table}
 \caption{Errors in $H^1_D(0,5)$  for \eqref{CaMoEx} using $M=N+1$ for the Campbell-Moore example. $n$ equidistant grid points and $M$ Gaussian collocation points have been used}\label{T.CaMoe}
\begin{center}
\begin{tabular}{cccccc}
\hline 
$n$ & $N=1$ & $N=2$ & $N=3$ & $N=4$ & $N=5$\tabularnewline
\hline 
10 & 3.32e+1 & 4.53e+0 & 3.82e-1 & 7.02e-2 & 1.47e-3\tabularnewline
20 & 3.32e+1 & 7.51e-1 & 1.02e-1 & 1.26e-2 & 1.24e-4\tabularnewline
40 & 3.32e+1 & 3.03e-1 & 3.14e-2 & 2.52e-3 & 1.30e-5\tabularnewline
80 & 3.32e+1 & 1.80e-1 & 1.22e-2 & 5.45e-4 & 1.54e-6\tabularnewline
160 & 3.32e+1 & 1.17e-1 & 5.67e-3 & 1.25e-4 & 1.20e-6\tabularnewline
320 & 3.32e+1 & 7.95e-2 & 2.73e-3 & 1.25e-4 & 1.20e-6\tabularnewline
\hline 
\end{tabular}
\end{center}
\end{table}

\begin{table}
 \caption{Order estimation for \eqref{CaMoEx} using $M=N+1$ for the Campbell-Moore example. $n$ equidistant grid points and $M$ Gaussian collocation points have been used. The row ``theory'' contains the expected orders. Note that Theorem~\ref{t.pi} is only valid for $N\geq 5$}\label{T.CaMoo}
\begin{center}
\begin{tabular}{cccccc}
\hline 
$n$ & $N=1$ & $N=2$ & $N=3$ & $N=4$ & $N=5$\tabularnewline
\hline 
20 & 0.0 & 2.6 & 1.9 & 2.5 & 3.6\tabularnewline
40 & 0.0 & 1.3 & 1.7 & 2.3 & 3.3\tabularnewline
80 & 0.0 & 0.7 & 1.4 & 2.2 & \textcolor{blue}{3.1}\tabularnewline
160 & 0.0 & 0.6 & 1.1 & 2.1 & 0.6\tabularnewline
320 & 0.0 & 0.6 & 1.1 & 0.0 & 0.0\tabularnewline
\hline 
theory &  & (0) & (1) & (2) & 3\tabularnewline
\hline 
\end{tabular}
\end{center}
\end{table}

\section{Multilevel approach}\label{s.Iteration}

We use $N$ and $s$ as previously agreed, that is $N-\mu+1>s>\mu-1/2>0$. 
Given an additional constant $q$ with $0<q<1$ we now deal with a sequence of partitions $\pi_{i}\in \mathcal M_{[r]}$, 
\[
\pi_{i}: a=t_{0}^{[\pi_{i}]}<t_{1}^{[\pi_{i}]}<\cdots<t_{n_{\pi}}^{[\pi_{i}]}=b, \quad \text{with maximal stepsize}\quad qh_{\pi_{i}}=h_{\pi_{i+1}},\; i\geq 0,
\]
such that the associated ansatz spaces are nested,
\[
 X_{\pi_{0}}\subset X_{\pi_{1}}\subset\cdots\subset X_{\pi_{i}}\subset X_{\pi_{i+1}}\subset\cdots
\]
and 
$h_{\pi_{i}}\rightarrow 0$  if $i\rightarrow\infty$. Let $\pi_{0}$ be fine enough for Lemma \ref{l.basic} and Theorem \ref{t.pi} to hold. This means that 
\begin{align*}
 \Gamma_{\pi_{0}}(\rho_{\pi_{0}}h_{\pi_{0}}^{-1/2}L+2\hat Lh_{\pi_{0}}^{N-1/2})  \leq\frac{1}{2},\quad \text{and}\quad h_{\pi_{0}}^{N-\mu+1-s}\leq\frac{1}{2}\frac{c_{\rho}}{c_{*}},
\end{align*}
to ensure the applicability of Lemma \ref{l.basic} and to make the iterations on the level $\pi_{0}$ to stay in $\bar{\mathfrak B}(x_{*},\rho_{\pi_{0}})\cap X_{\pi_{0}}$. Both conditions are satisfied correspondingly a fortiori on the further levels due to the smaller stepsizes $h_{\pi_{i}}$. In the consequence, Theorem \ref{t.pi} applies on each level, i.e.,
for $x_{0}^{[\pi_{i}]}\in \bar{\mathfrak B}(x_{*},\rho_{\pi_{i}})\cap X_{\pi}$ the sequence 
\begin{align}
 x_{k+1}^{[\pi_{i}]}&=x_{k}^{[\pi_{i}]}+z_{k+1}^{[\pi_{i}]},\label{xleveli}\\
 z_{k+1}^{[\pi_{i}]}&=\argmin\{\,\rVert \mathcal F'(x_{k}^{[\pi_{i}]})z+\mathcal Fx_{k}^{[\pi_{i}]}\lVert^{2}: z\in X_{\pi}\}
 =-(\mathcal F'(x_{k}^{[\pi_{i}]})U_{\pi})^{+}\,\mathcal Fx_{k}^{[\pi_{i}]},\quad k\geq 0.\label{zleveli}
\end{align}
remains in $\bar{\mathfrak B}(x_{*},\rho_{\pi_{i}})\cap X_{\pi_{i}}$ and there exists a number  $k_{\pi_{i}}\in \N$ such that   $\left(\frac{1}{2}\right)^{k+1}\leq \frac{c_{*}}{c_{\rho}}h_{\pi_{i}}^{N-\mu+1-s}$ for all $k\geq k_{\pi_{i}}$, and hence 
\begin{align}\label{estimate_i}
 \rVert x_{k+1}^{[\pi_{i}]}-x_{*}\lVert\leq 3c_{*} h_{\pi_{i}}^{N-\mu+1},\quad k\geq k_{\pi_{i}}.
\end{align}
Since the ansatz spaces are nested, $x_{k_{\pi_{i}+1}}^{[\pi_{i}]}$ belongs to $X_{\pi_{i+1}}$. Replacing the condition 
$h_{\pi_{i}}^{N-\mu+1-s}\leq\frac{1}{2}\frac{c_{\rho}}{c_{*}}$ by the stronger one
\begin{align}\label{zusatzbed}
 h_{\pi_{i}}^{N-\mu+1-s}\leq\frac{1}{3}q^{s}\frac{c_{\rho}}{c_{*}}
\end{align}
yields
\begin{align*}
 3c_{*}h_{\pi_{i}}^{N-\mu+1}\leq 3c_{*}\frac{1}{3}q^{s}\frac{c_{\rho}}{c_{*}}h_{\pi_{i}}^{s} =c_{\rho}q^{s}h_{\pi_{i}}^{s}=c_{\rho}h_{\pi_{i+1}}^{s}=\rho_{\pi_{i+1}}.
\end{align*}
Then $x_{k_{\pi_{i}}+1}^{[\pi_{i}]}$ belongs to $\bar{\mathfrak B}(x_{*},\rho_{\pi_{i+1}})\cap X_{\pi_{i+1}}$ and we are allowed to choose at the next level
\begin{align}\label{connect}
 x_{0}^{[\pi_{i+1}]}:=x_{k_{\pi_{i}+1}}^{[\pi_{i}]}.
\end{align}
We summarize our result:
\begin{theorem}\label{t.main}
Let $\mathcal Fx=0$ denote the operator formulation from Section \ref{s.Hilbert_setting} associated with the BVP (\ref{DAE}),(\ref{BC}), $\mathcal Fx_{*}=0$, $\ker \mathcal F'(x_{*})=\{0\}$, and $x_{*}$ be sufficiently smooth for (\ref{alpha}).\\
Let (\ref{N_and_s}) be given and $0<q<1$.

Let the sequence of partitions $\pi_{i}\in \mathcal M_{[r]}$, $i\geq 0$, be such that the ansatz spaces are nested and the maximal stepsizes are related by $qh_{\pi_{i}}=h_{\pi_{i+1}}$.
Let the the mesh $\pi_{0}$ be sufficiently fine, 
\begin{align*}
 \Gamma_{\pi_{0}}(\rho_{\pi_{0}}h_{\pi_{0}}^{-1/2}L+2\hat Lh_{\pi_{0}}^{N-1/2})  \leq\frac{1}{2},\quad
 \text{and}\quad h_{\pi_{0}}^{N-\mu+1-s}\leq \frac{1}{3}q^{s}\frac{c_{\rho}}{c_{*}}.
\end{align*}
Then the iteration (\ref{xleveli}),(\ref{zleveli}),(\ref{connect}), with the initial guess $x_{0}^{[\pi_{0}]}\in \bar{\mathfrak B}(x_{*},\rho_{\pi_{0}})\cap X_{\pi_{0}}$ is well defined and yields
\begin{align}\label{convergence}
 \rVert x_{k_{\pi_{i}+1}}^{[\pi_{i}]}-x_{*}\lVert\leq 3c_{*} h_{\pi_{i}}^{N-\mu+1}=3c_{*} h_{\pi_{0}}^{N-\mu+1}\bigl(q^{N-\mu+1}\bigr)^{i+1}\rightarrow 0 \quad (i\rightarrow \infty).
\end{align}
\end{theorem}

\section{Remarks and conclusions}

We have presented and investigated a nonlinear least-squares method for approximating higher index
differential-algebraic equations. The idea consists of discretizing the preimage space $H^1_D$ by
piecewise polynomials and to form an overdetermined collocation system to determine an approximating solution. The resulting overdetermined system is solved in a least-squares sense. In the numerical experiments, the method behaved very well despite its simplicity.
In particular, the method is not much more expensive than the standard collocation method applied to explicit
ordinary differential equations and index-1 differential-algebraic equations.

The main tool both for the convergence proof and for the numerical solution of the discretized problems
is a variant of the Newton method.
For a large class of nonlinear index-$\mu$ tractable equations, this method 
applied to the discretized system is shown to deliver appropriate approximations provided that the polynomial order is large enough. The numerical experiments indicate, however, that the strong condition on the
polynomial order does not seem to be necessary. In particular, the order of convergence corresponds to
that of linear index-$\mu$ tractable differential-algebraic equations. So the present result should be
considered as a first
step towards a theoretical foundation of the method.

\begin{remark}
 \begin{enumerate}[(i)]
  \item Under the conditions of Theorem~\ref{t.pi} we could not show that the sequence $\{x_k\}$ converges.
\item If there is a minimizer $x_{\pi,*}$ of \eqref{lscF} in ${\mathfrak B}(x_{*},\rho_{\pi})\cap X_{\pi}$,
then it holds $(\mathcal{F}'({x}_{\pi,*})\mathcal{U}_{\pi})^{+}\mathcal{F}({x}_{\pi,*})=0$.

Since $\bar{\mathfrak B}(x_{*},\rho_{\pi})\cap X_{\pi}$ is compact, the sequence $\{x_k\}$ has a
convergent subsequence. However, we were not able to show that, for an accumulation
point $\hat{x}_{\pi}$, it holds $(\mathcal{F}'(\hat{x}_{\pi})\mathcal{U}_{\pi})^{+}\mathcal{F}(\hat{x}_{\pi})=0$.
\item In the context of regularization
methods for nonlinear illposed problems, the so-called Scherzer, or tangential cone, condition is often used
\cite{Sch95,HaNeSch95,KaltOffter}. However, in the context of differential-algebraic equations,
this conditions requires very hard conditions on the structure of the system. Therefore, it is of minor
use here.
 \end{enumerate}
\end{remark}

\appendix

\section{An auxillary result}

The convergence proof for the Gauss-Newton method requires an estimation of the norm and distance of Moore-Penrose inverses of derivatives of a nonlinear operator. In the case of finite dimensional spaces, such results are well-known and can be found, for example, in \cite{LaHa74}. However, we need similar statements in
the case of infinite dimensional spaces. This appendix provides the necessary lemmas.

Let $X$ and $Y$ be Hilbert spaces (not necessarily finite dimensional)
and $A:X\rightarrow Y$ a linear and compact operator. Both operators
$A^{\ast}A$ and $AA^{\ast}$ are selfadjoint compact operators. Their
spectra consist only of nonnegative eigenvalues with finite multiplicity (with
the possible exception of $\lambda=0$). If the eigenvalues have an accumulation point,
then it is 0. The nonzero eigenvalues
are identical (even with respect to their multiplicity) for both $A^{\ast}A$
and $AA^{\ast}$. Let the nonzero eigenvalues be sorted according to
\[
\lVert A^{\ast}A\rVert=\lambda_{1}\geq\lambda_{2}\geq\cdots>0.
\]
Let then $\{u_{i}\}\subset X$ and $\{v_{i}\}\subset Y$ be a complete
orthonormal system of eigenvalues\footnote{The systems are not necessarily complete in $X$ and $Y$, respectively!}
for the operators $A^{\ast}A$ and $AA^{\ast}$,
\[
\lambda_{i}u_{i}=A^{\ast}Au_{i},\quad\lambda_{i}v_{i}=AA^{\ast}v_{i}.
\]
We set $\sigma_{i}=\sqrt{\lambda_{i}} > 0$. This provides us with
\[
\sigma_{i}u_{i}=A^{\ast}v_{i},\quad\sigma_{i}v_{i}=Au_{i}.
\]
The system $\{\sigma_{i},u_{i},v_{i}\}$ is called a singular system
of $A$ with the singular values $\sigma_{i}$. In particular, we
have the representations
\begin{align*}
Ax & =\sum_{i}(x,u_{i})v_{i},\:x\in X,\\
A^{\ast}y & =\sum_{i}(y,v_{i})u_{i},\:y\in Y.
\end{align*}
Here, $(\cdot,\cdot)$ denotes the scalar products in $X$ and $Y$,
respectively. Note that these sums can be both finite and infinite.

We are interested in perturbation results for the singular values
of an operator $A$. The following lemma is proven in \cite[Corollary VI.1.6]{GoGoKa90}.
\begin{lem}\label{cor:SV}
Let $X$ and $Y$ be Hilbert spaces. Let $A,B:X\rightarrow Y$
be compact linear operators. Let $\sigma_{i}(A)$, $i=1,\ldots,\nu(A)$
and $\sigma_{i}(B)$, $i=1,\ldots,\nu(B)$ be the singular values
of $A$ and $B$, respectively.\footnote{Both $\nu(A)$ and $\nu(B)$ may be finite
or infinite.} Assume without loss of generality
$\nu(A)\leq\nu(B)$. Then it holds
\begin{align*}
\lvert\sigma_{i}(A)-\sigma_{i}(B)\rvert & \leq\rVert A-B\rVert,\quad i=1,\ldots,\nu(A),\\
\sigma_{i}(B) & \leq\lVert A-B\rVert,\quad i=\nu(A)+1,\ldots,\nu(B).
\end{align*}
\end{lem}
The next step consists of the establishement of bounds for the Moore-Penrose
pseudoinverse. For compact operators as above with the singular system
$\{\sigma_{i},u_{i},v_{i}\}$ it has the representation
\[
x=A^{+}y=\sum_{i}\sigma_{i}^{-1}(y,v_{i})u_{i}
\]
for all $y\in\dom(A^{+})$. An immediate consequence is:
\begin{enumerate}[(i)]
\item  $\lVert A\rVert=\sigma_{1}$.
\item $A^{+}$ is bounded if and only if the number of singular values
is finite.
In that case it holds $\lVert A^{+}\rVert=\sigma_{\nu(A)}^{-1}.$
\end{enumerate}
The following lemma presents modifications of \cite[Theorem (8.15)]{LaHa74}.

\begin{lem}\label{ST}
Let $A,B:X\rightarrow Y$ be compact linear operators acting in the Hilbert spaces $X,Y$.
\begin{enumerate}[(i)]
 \item
Assume  $\dim X<\infty$ and $\nu(B)\leq \nu(A)=r$, $r\geq 1$. Moreover, let
\[
\lVert A^{+}\rVert\lVert B-A\rVert<1
\]
to hold and set
$\epsilon=\lVert B-A\rVert$. Then it holds $\nu(B)=\nu(A)$ and
\[
\lVert B^{+}\rVert\leq\frac{\lVert A^{+}\rVert}{1-\lVert A^{+}\rVert\lVert B-A\rVert}=\frac{1}{\sigma_{\nu(A)}-\epsilon}
\]
where $\sigma_{\nu(A)}$ is the smallest singular value of $A$.
\item Assume that $X$ decomposes in $X=X_{f}\oplus X_{f}^{\bot}$, $\dim X_{f}< \infty$, and 
\begin{align*}
 X_{f}^{\bot}=\ker A\subseteq \ker B,\quad \im B\subseteq\im A,\quad \lVert A^{+}\rVert\lVert B-A\rVert<1.
 \end{align*}
 Then it follows that
 \[
\lVert B^{+}\rVert\leq\frac{\lVert A^{+}\rVert}{1-\lVert A^{+}\rVert\lVert B-A\rVert}.
\]
\end{enumerate}
\end{lem}
\begin{proof}
It holds $\lVert A^{+}\rVert=\sigma_{\nu(A)}^{-1}$ such that,
by assumption, $\sigma_{i}-\epsilon>0$, $i=1,\ldots,\nu(A)$. Hence,
$B=A+(B-A)$ has at least $\nu(A)$ nonvanishing singular values because
of Lemma~A.\ref{cor:SV}. Hence, $\nu(B)\geq \nu(A)$. Together with the
assumption, this provides $\nu(B)=\nu(A)$. Consequently, $\lVert B^{+}\rVert\leq1/(\sigma_{\nu(A)}-\epsilon)$.
(ii) is a consequence of (i).
\end{proof}

\bibliographystyle{plain} 
\bibliography{hm}

\end{document}